\documentclass[11pt]{amsart}
\usepackage{amssymb}
\usepackage{color}

\usepackage[hmargin={20mm,20mm},vmargin={15mm,20mm}]{geometry}
\usepackage{geometry}
\usepackage[dvips]{graphicx}
\usepackage{graphicx}

\bibstyle{plain}
\newtheorem{theorem}{Theorem}
\newtheorem{example}[theorem]{Example}
\newtheorem{problem}[theorem]{Problem}
\newtheorem{lemma}[theorem]{Lemma}

\newtheorem{remark}[theorem]{Remark}

\newcommand{\F}{\mathcal F}
\newcommand{\G}{\mathcal G}
\newcommand{\K}{\mathbb K}

\newcommand{\N}{\mathbb N}

\newcommand{\R}{\mathbb R}

\newcommand{\mH}{\mathcal{H}}

\newcommand{\on}{\operatorname}
\subjclass[2010]{}

\author{Patrycja Jaros}

\address{Division of Dynamics, \L\'od\'z University of Technology, 
\L\'od\'z, Poland}
\email {patrycja.kuzma@p.lodz.pl}

\author{\L ukasz Ma\'slanka}

\address{Institute of Mathematics, \L\'od\'z University of Technology, 
\L\'od\'z, Poland}
\email {maslana0189@gmail.com}

\author{Filip Strobin}

\address{Institute of Mathematics, \L\'od\'z University of Technology, W\'olcza\'nska 215, 93-005
\L\'od\'z, Poland}
\email {filip.strobin@p.lodz.pl}

\title[]{algorithms generating images of attractors of generalized iterated function systems}
\date{}

\begin{document}
\maketitle
\begin{abstract}
The paper is devoted to searching algorithms which will allow to generate images of attractors of \emph{generalized iterated function systems} (GIFS in short), which are certain generalization of classical iterated function systems, defined by Mihail and Miculescu in 2008, and then intensively investigated in the last years (the idea is that instead of selfmaps of a metric space $X$, we consider mappings form the Cartesian product $X\times...\times X$ to $X$).\\ 
Two presented algorithms are counterparts of classical \emph{deterministic algorithm} and so-called \emph{chaos game}. The third and fourth one one is fitted to special kind of GIFSs - to \emph{affine} GIFS, which are, in turn, also investigated. 
\end{abstract}

\section{Introduction}
If $X$ is a metric space and $\F=\{f_1,...,f_n\}$ is a finite family of continuous selfmaps of $X$, then by the same letter $\F$ we will also denote the function $\F:\K(X)\to\K(X)$ defined by 
$$
\F(K):=f_1(K)\cup...\cup f_n(K),
$$
where $\K(X)$ denotes the family of all nonempty and compact subsets of $X$ (and is considered as a metric space with the classical Hausdorff-Pompeiu metric, which will be denoted by the letter $H$).\\
The following Hutchinson-Barnsley theorem (first proved by Hutchinson \cite{H}, then popularized by Barnsley \cite{B}), is one of the milestones for the fractal theory.
\begin{theorem}\label{f1}
Assume that $(X,d)$ is a complete metric space and $\F=\{f_1,...,f_n\}$ is a finite family of Banach contractions of $X$ (i.e., selfmaps of $X$ with Lipschitz constants less then $1$). Then there is a unique set $A_\F\in \K(X)$ such that $$A_\F=\F(A_\F)=f_1(A_\F)\cup...\cup f_n(A_\F)$$
Moreover, for every $K\in\K(X)$, the sequence of iterations $\F^{(k)}(K)$ converges to $A_\F$ with respect to the Hausdorff-Pompeiu metric.
\end{theorem}
Note that in the above theorem, instead of Banach contractions, there can be consider many weaker types of contractive mappings (later we will discuss this topic broader).\\
In the above frame, family of mappings $\F$ is called \emph{iterated function system}, and a set $A_\F$ - the \emph{fractal generated by $\F$}. Also a compact set $A$ is called \emph{a Hutchinson-Barnsley fractal}, if it is a fractal generated by some IFS.\\
It turnes out that many classical fractals, like the Cantor ternary set, the Sierpiński gasket etc., are H-B fractals. Also many objects from the nature (like trees, clouds etc.) can be modelled as H-B fractals.\\
In connection with the above statement, a natural problem is:
\begin{problem}\label{f2}
Given an IFS $\F$ on an Euclidean space $\R^d$, how to make an image of its fractal $A_\F$?
\end{problem} There are two main algorithms that can be easily adjusted to computer programs:\\
The first one, called the \emph{deterministic algorithm}, bases on the second part of Theorem \ref{f1}: we choose a compact set $K$, then we find $K_1:=\F(K)$, then $K_2:=\F(K_1)$ and so on. By Theorem \ref{f1}, sets $K_n$ are better and better approximations of the fractal $A_\F$.\\
The second one, called the \emph{chaos game algorithm}, goes in the following way: we choose randomly any point $x_0\in X$, then we choose randomly $i_1\in\{1,...,n\}$, and put $x_1:=f_{i_1}(x_0)$. Then we choose randomly $i_2\in\{1,...,n\}$, and put $x_2:=f_{i_2}(x_1)$ and so on. As a consequence, we obtain a sequence $x_0,x_1,x_2,...$, and it turns out that sets $\{x_k,...,x_N\}$, where $k<N$ are appropriately big, are good approximations of the fractal $A_\F$.\\
In this paper we are going to consider Problem \ref{f2} for the fractals generated by \emph{generalized iterated function systems} introduced by Mihail and Miculescu, then intensively investigated by them and also Strobin and Swaczyna, and Secelean - see for example \cite{M},\cite{M1},\cite{MM1},\cite{MM2},\cite{MS}, \cite{SS1} and \cite{SS2}. We now recall the notion of GIFSs.\\
Let $(X,d)$ be a metric space and $m\in\N$. By $X^m$ we denote the Cartesian product of $m$ copies of $X$. We endow $X^m$ with the maximum metric $d_m$:
$$
d_m((x_1,...,x_m),(y_1,...,y_m)):=\max\{d(x_1,y_1),...,d(x_m,y_m)\}
$$
We say that $f:X^m\to X$ is a \emph{generalized Matkowski contraction (of order $m$)}, if for some nondecreasing function $\varphi:[0,\infty)\to[0,\infty)$ such that for any $t>0$, the sequence of iterates $\varphi^{(k)}(t)\to 0$, we have that
$$
d(f(x),f(y))\leq \varphi(d_m(x,y)),\;\;\;x,y\in X^m
$$ 
Note that if $m=1$, then we get a \emph{Matkowski contraction} \cite{Ma}, and it is well known that Theorem \ref{f1} can be strengthened by considering IFSs consisting of Matkowski contractions (clearly, each Banach contraction is Matkowski, but the converse need not be true).\\
Finally, a finite family $\F=\{f_1,...,f_n\}$ of generalized Matkowski contractions of order $m$ is called a \emph{generalized iterated function system of order $m$} (GIFS in short).\\
Also by $\F$ we will denote the mapping $\F:\K(X)^m\to \K(X)$ defined by
$$
\F(K_1,...,K_m):=f_1(K_1\times...\times K_m)\cup...\cup f_n(K_1\times...\times K_m)
$$
The following result is an extension of Theorem \ref{f1} (see \cite{MM1},\cite{MM2},\cite{SS1}, \cite{SS2}): 
\begin{theorem}\label{f3}
Let $X$ be a complete metric space and $\F$ be a GIFS of order $m$. Then there exists a unique set $A_\F\in\K(X)$ such that
$$
A_\F=\F(A_\F,...,A_\F)=f_1(A_\F\times...\times A_\F)\cup...\cup f_n(A_\F\times...\times A_\F)
$$
Moreover, for every $K_0,...,K_{m-1}\in\K(X)$, the sequence $(K_k)$, defined by $K_{k+m}:=\F(K_k,...,K_{m+k-1})$, converges to $A_\F$ with respect to the Hausdorff-Pompeiu metric.
\end{theorem}
In the above frame, the set $A_\F$ is called the \emph{fractal generated by $\F$}. Also, a compact set $A\subset X$ is called \emph{a Hutchinson-Barnsley generalized fractal}, if it is a fractal generated by some GIFS.\\
Note that GIFSs give us some new fractals - there are sets which are attractors of some GIFSs, but are not attractors of any IFSs (even consisting of Matkowski contractions) and there are compact sets which are not attractors of any GIFSs - see \cite{S}. Also see \cite{MM2} for another example.\\
We will also need some facts which involves the notion of a \emph{code space} for GIFSs (for detailed discussion see \cite{M1}, \cite{SS2} and \cite{SRum}). The constructions are a bit technically complicated, but, in fact, they are natural counterparts of classical case of IFSs.\\
So assume that $\F=\{f_1,...,f_n\}$ is a GIFS of order $m$, on a space $X$.\newline
Define $\Omega _{1},\Omega _{2},\ldots 
$ by the following inductive formula:
\begin{equation*}
\Omega _{1}:=\{1,...,n\} 
\end{equation*}%
\begin{equation*}
\Omega _{k+1}:=\underbrace{\Omega _{k}\times \ldots \times \Omega _{k}}_{m%
\mbox{ times}}\;\;\;\mbox{for }k\geq 1
\end{equation*}%
\indent Then for every $k\in\mathbb{N}$, let 
\begin{equation*}
{}_{k}\Omega :=\Omega _{1}\times \ldots \times \Omega _{k}
\end{equation*}%
and define 
\begin{equation*}
\Omega _{<}:=\bigcup_{k\in\mathbb{N}}{}_{k}\Omega 
\end{equation*}%
and 
\begin{equation*}
\Omega :=\Omega _{1}\times \Omega _{2}\times \Omega _{3}\times \ldots =%
\underset{i\in\mathbb{N}}{\Pi }\Omega _{i}.
\end{equation*}

The space $\Omega $ is called a \emph{code space for $\F$}.

\begin{remark}
\emph{\ In the case $m=1$ we have 
\begin{equation*}
\Omega =\{1,...,n\}^\N
\end{equation*}%
so $\Omega$ is the standard code space for an iterated function systems consisting of $n $
mappings (see \cite{H},\cite{B}).}
\end{remark}
Now we will define the families of mappings $\mathcal{F}^{k}$, $k\in\mathbb{N}$,
inductively with respect to $k$.\newline
If $k>1$ and 
\begin{equation*}
\alpha =(\alpha ^{1},\ldots ,\alpha ^{k})=(\alpha ^{1},(\alpha
_{1}^{2},\ldots ,\alpha _{m}^{2}),\ldots ,(\alpha _{1}^{k},\ldots ,\alpha
_{m}^{k}))\in {}_{k}\Omega 
\end{equation*}%
then for any $i=1,\ldots ,m$, we set 
\begin{equation}
\alpha (i):=(\alpha _{i}^{2},\alpha _{i}^{3},\ldots ,\alpha _{i}^{k}).
\label{n1}
\end{equation}%
Clearly, $\alpha (i)\in {}_{k-1}\Omega $.\newline
If $\alpha \in \Omega $, we define $\alpha (i)\in \Omega $ in an analoguos
way.\newline
Also, define spaces $X_1,X_2,...$ be the following inductive formula:
\begin{equation}\label{filip1}X_1:=\underbrace{X\times \ldots \times X}_{m%
\mbox{ times}}\end{equation}
\begin{equation}\label{filip2}
X_{k+1}:=\underbrace{X_{k}\times \ldots \times X_{k}}_{m%
\mbox{ times}}\end{equation}
We are ready to define the families $\mathcal{F}^{k}$, $k\in\mathbb{N}$.\newline
Define 
\begin{equation*}
\mathcal{F}^{1}:=\mathcal{F}
\end{equation*}%
and observe that (since $_1\Omega=\Omega_1=\{1,...,n\}$) we can write $\mathcal{F}^{1}=\{f_{\alpha }:\alpha \in \;_{1}\Omega \}$
and each $f_{\alpha }\in \mathcal{F}^{1}$ is a function from $X_{1}$ to $X$.

Assume that we have already defined $\mathcal{F}^k=\{f_\alpha:\alpha\in\;_k%
\Omega\}$ such that each $f_\alpha\in\mathcal{F}^k$ is a function from $X_k$
to $X$. Then for every $\alpha=(\alpha^1,\ldots ,\alpha^k,\alpha^{k+1})\in{}%
_{k+1}\Omega$, let $f_\alpha:X_{k+1}\to X$ be defined by 
\begin{equation*}
f_\alpha(x_1,\ldots ,x_m)=f_{\alpha^1}(f_{\alpha(1)}(x_1),\ldots
,f_{\alpha(m)}(x_m)),
\end{equation*}
for each $(x_{1},...,x_m) \in X_{k}\times\dots\times X_k=X_{k+1}$.\newline
Then set 
\begin{equation*}
\mathcal{F}^{k+1}:=\{f_\alpha:\alpha\in\;_{k+1}\Omega\} 
\end{equation*}
Finally, define $\F^{<}:=\bigcup_{k\in\N}\F^k=\{f_\alpha:\alpha\in\Omega_{<}\}$.

\begin{remark}
\emph{\ Clearly, in the case when $m=1$, if $\alpha=(\alpha^1,...,\alpha^k)%
\in\;_k\Omega$, then $f_\alpha=f_{\alpha_1}\circ...\circ f_{\alpha_k}$,
hence defined families of mappings are natural generalizations of
compositions.}
\end{remark}
If $D\in\K(X)$, then we define sets $D_1,D_2,...$ in a similar way as spaces $X_1,X_2,...$ in (\ref{filip1}) and (\ref{filip2}), and then, for every $k\in\N$ and $\alpha\in\;_k\Omega$, we define
$$
D_\alpha:=f_\alpha(D_k)
$$
Also, we set $A_\alpha:=(A_\F)_\alpha$ for $\alpha\in\Omega_<$.\\
The following result can be found in \cite{SS2}.
\begin{lemma}\label{l1}
Let $\F$ be a GIFS of order $m$ on a complete metric space $X$. Then
\begin{itemize}
\item[(a)] For every $D\in\K(X)$, $\lim_{k\to\infty}\max\{\operatorname{diam}D_\alpha:\alpha\in\;_k\Omega\}=0$.
\item[(b)] For every $k\in\N$, $A_\F=\bigcup_{\alpha\in\;_k\Omega}A_\alpha$.
\end{itemize}
\end{lemma}
We will also need the following
\begin{lemma}\label{filip4}
Let $\F$ be a GIFS of order $m$ on a complete metric space $X$.
For every closed and bounded set $D\subset X$ and $k\in\N$, define $$K_k=\bigcup_{\alpha\in\;_k\Omega}f_\alpha(D_k)$$
where $D_1,D_2,...$ are defined as in (\ref{filip1}) and (\ref{filip2}). Then $K_k\to A_\F$ with respect to the Hausdorff-Pompeiu metric.
\end{lemma}
\begin{proof}
By the previous Lemma, for every $k\in\N$, we have $A_\F=\bigcup_{\alpha\in\;_k\Omega}A_\alpha$. Hence, using a standard property of the Hausdorff-Pompeiu metric, we get
$$
H(A_\F,K_k)=H\left(\bigcup_{\alpha\in\;_k\Omega}A_\alpha,\bigcup_{\alpha\in\;_k\Omega}f_\alpha(D_k)\right)\leq
\max\left\{H\left(A_\alpha,f_\alpha(D_k)\right):\alpha\in\;_k\Omega\right\}\leq\varphi^{(k)}(H(A_\F,D))
$$
where $\varphi$ is a witness to the fact that $\F$ is a GIFS (clearly, by choosing the maximum function, we can assume that we have one function $\varphi$ for all $f_1,...,f_n$). Note that the last inequality can be proved in a standard way - the case $k=1$ can be proved as \cite[Theorem 3.7]{SS1}, and then we can proceed by induction.\\
Now, using the known (and easy) fact that $\varphi^{(k)}(t)\to 0$ for all $t\geq 0$, we get the thesis.
\end{proof}
In the next section we will introduce the counterpart of deterministic algorithm for GIFS. Section 3 is devoted to counterpart of chaos game algorithm for GIFSs. Finally, in Section 4 we will study affine GIFSs, and introduce (deterministic) algorithm for such GIFSs.
\section{Deterministic Algorithm for GIFSs}
Assume that $\F$ is a GIFS of order $m$ on a complete metric space $X$. The pseudocode for {deterministic algorithm} for $\F$ is the following
\begin{center}\textbf{Pseudocode for deterministic algorithm for GIFSs}\end{center}
Initially chosen compact sets: $K_0,...,K_{m-1}\subset X$.\\
Initially defined objects: constant: $m$, mappings: $f_1,...,f_n$, variables: $i,D_0,...,D_{m-1}$.\\
Initial values: $D_0:=K_0$,...,$D_{m-1}:=K_{m-1}$.\\
Main loop: 

$\;\;\;\;\;\;\;\;\;\;\;\;\;$ $K:=\F(D_0,...,D_{m-1})$

$\;\;\;\;\;\;\;\;\;\;\;\;\;$  For $i$ from $0$ to $m-2$

$\;\;\;\;\;\;\;\;\;\;\;\;\;\;\;\;\;\;\;$ $D_i:=D_{i+1}$

$\;\;\;\;\;\;\;\;\;\;\;\;\;$  $D_{m-1}:=K$\\


By Theorem \ref{f3}, sets $K$ are closer and closer to the fractal $A_\F$.\\
Now we give some examples:
\begin{example}\label{e1}\emph{
Consider the GIFSs $\F=\{f_1,f_2\}$ and $\G=\{g_1,g_2\}$  on $\R^2$, where, for $x=(x_1,x_2)$ and $y=(y_1,y_2)$, we have\\
$\;$\\
$
f_1(x,y):=(0,1x_1+0,15y_1+0,04y_2\;\mbox{\textbf{;}}\;0,16x_2-0,04y_1+0,15y_2+1,6)
$\\
$
f_2(x,y):=(0,1x_1-0,15x_2-0,1y_1+0,15y_2+1,6\;\mbox{\textbf{;}}\;0,15x_1+0,15x_2+0,15y_1+0,07)
$\\$\;$\\
$
g_1(x,y):=(0,05x_1+0,02y_1+0,1y_2+0,635\;\mbox{\textbf{;}}\;0,1x_1+0,2x_2+0,08y_1+0,15y_2+0,5)
$\\
$
g_2(x,y):=(0,15x_1+0,05y_1+0,1y_2+0,5\;\mbox{\textbf{;}}\;0,15x_1+0,15x_2+0,45)
$\\$\;$\\
It can be easily seen that $\F$ and $\G$ are GIFSs, indeed.\\
Using deterministic algorithm for these GIFSs, we get the following images:}
\begin{center}
\includegraphics{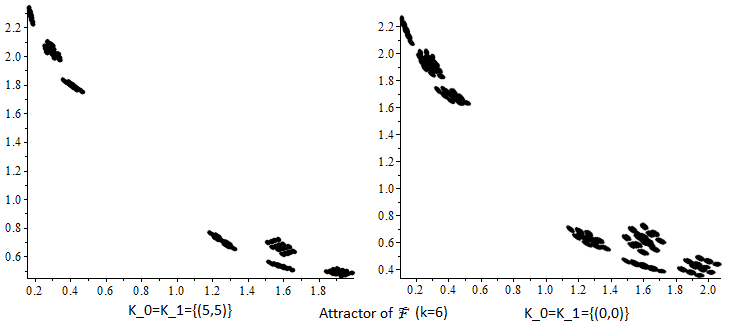}
\end{center}
\begin{center}
\includegraphics{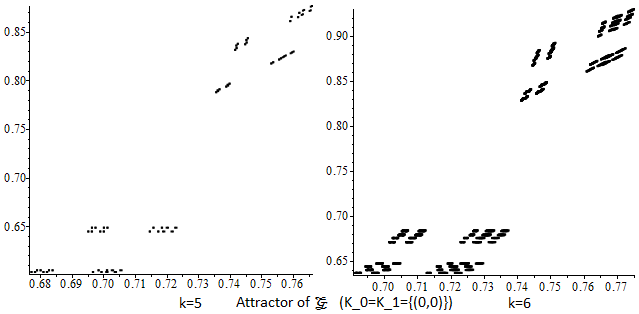}
\end{center}
\end{example}
\begin{remark}\emph{
Observe that if $\F$ is a GIFS on $X$ and $K_0,...,K_{m-1}\in\K(X)$ are such that $K_i\subset A_\F$ for $i=0,...,m-1$, then each $K_k\subset A_\F$, where $(K_k)$ is defined as in Theorem \ref{f3}. Hence, if we assure our start sets are subsets of the attractor, then we can draw all the points from each step.
}\end{remark}
\begin{remark}\emph{
Assume that $\F$ is a GIFS on $X$ and $K_0,...,K_{m-1}\in\K(X)$ such that the cardinality $card(K_i)\leq M$ for $i=0,...,m-1$. By an inductive argument, we can calculate that for every $k\in\N$,
$card(K_{m+k})\leq n^{2^k}M^{2^k(m-1)+1}$, where $(K_k)$ is defined as in Theorem \ref{f3}. Hence we can easily estimate the maximal required memory for calculating $K_{m+k}$ (we need to remember  
$
card(K_{k})+...+card(K_{k+m-1})
$
points).
}\end{remark}

\begin{remark}
\emph{Note that if $\F$ is a GIFS on $X$ such that $c:=\max\{Lip(f):f\in\F\}<1$, and $K_0=...=K_{m-1}=K$ for some $K\in\K(X)$, then, using standard properties of the Hausdorff-Pompeiu metric, for every $k\geq 0$, the distance
$$
H(K_k,A_\F)\leq c^{\lfloor \frac{k}{m} \rfloor}H(K,A_\F)
$$
where $(K_k)$ is defined as in Theorem \ref{f3}. This gives some control on the speed of convergence of the sequence $(K_k)$ to the attractor.
}
\end{remark}
\begin{remark}\emph{Observe that one can get the attractor $A_\F$ of a GIFS $\F$ by using another version of deterministic algorithm. Indeed, consider the mapping $\tilde{\F}:\K(X)\to\K(X)$ given by 
$$
\tilde{\F}:=f_1(K, ..., K) \cup ... \cup f_n(K, ...,K)
$$
Then $\tilde{\F}$ is a Matkowski contraction, so each sequence of iterations $(\tilde{\F}^k(K))$ converges to $A_\F$. Note that if $c:=\max\{Lip(f):f\in\F\}<1$, then $Lip(\tilde{F})\leq c$, so the speed of convergence of $(\tilde{\F}^k(K))$ is $\leq c^k$, so such an approach seems to be even more efficient.
}\end{remark}

\section{chaos game algorithm}
The counterpart of chaos game for GIFSs will be much more complicated, as the counterpart for composition for GIFSs is so.


At first we define a certain bijection $H:\N\times\N\to\N$:\\
Set $H(1,1):=1$. Assume that we have defined $H(j,k)$ for some $(j,k)$\\
-- if $j\on{mod} m\neq 0$, then we set $H(m^{k-1}\cdot j+1,1):=H(j,k)+1$,\\
-- if $j\on{mod} m=0$, then we set $H\left(\frac{j}{m},k+1\right):=H(j,k)+1$.
\begin{remark}\label{filip3}\emph{
Note that the definition of $H$ is correct - in Remark \ref{f3} below it can be seen how it works (for $m=2$) - the rows are connected to $"$levels$"$ $k$, and each row, the coefficient $j$ changes according to enumeration from the left. So, for example (when $m=2$), we have $H(1,1)=1$, $H(2,1)=2$, $H(1,2)=3$, $H(3,1)=4$, $H(4.1)=5$, $H(2,2)=6$, $H(1,3)=7$, $H(5,1)=8$,...,$H(2,3)=14$, $H(1,4)=15$,...}
\end{remark}
Now let $\F=\{f_1,...,f_n\}$ be a GIFS of order $m$ on a complete space $X$, and choose a sequence  $(\gamma_i)\in\{1,...,n\}^\N$ and $x_0\in X$. Then we can define the sequence $(x_i)\subset X$
by the following recursive formula:
$$x_1:=f_{\gamma_{1}}(x_0,...,x_0),\;x_2:=f_{\gamma_{2}}(x_0,...,x_0),\;...,\;x_{m}:=f_{\gamma_{m}}(x_0,...,x_0)$$
and if $i\geq m+1$ and $i=H(j,k)$, then:

$\;\;\;\;\;$if $k>1$, then $$x_i=x_{H(j,k)}:=
f_{\gamma_i}\left(x_{H(mj-m+1,k-1)},...,x_{H(mj,k-1)}\right)$$

$\;\;\;\;\;$and if $k=1$, then 
$$x_i:=f_{\gamma_i}(x',...,x'),$$ 

$\;\;\;\;\;$where $x':=x_{l}$ and $l:=H(j',k')$ is maximal number with the property that $l<i$ and $k'>1$. 

\begin{remark}\label{f3}\emph{
Observe that the definition of the sequence $(x_i)$ is, in fact, very natural. For example, if $m=2$, then we have:}\\

$k=1\;\;\;\;\;\;x_1\;\;\;\;\;\;\;\;\;\;\;\;\;\;\;\;x_2\;\;\;\;\;\;\;\;\;\;\;\;\;\;\;x_4\;\;\;\;\;\;\;\;\;\;\;\;\;\;\;x_5\;\;\;\;\;\;\;\;\;\;\;\;\;\;\;x_8\;\;\;\;\;\;\;\;\;\;\;\;\;\;\;x_9\;\;\;\;\;\;\;\;\;\;\;\;\;x_{11}\;\;\;\;\;\;\;\;\;\;\;\;\;\;\;x_{12}$\\

$k=2\;\;\;\;\;\;\;\;\;\;\;\;\;x_3=f_{\gamma_3}(x_1,x_2)\;\;\;\;\;\;\;\;\;\;\;\;x_6=f_{\gamma_6}(x_4,x_5)\;\;\;\;\;\;\;\;\;\;\;x_{10}=f_{\gamma_{10}}(x_8,x_9)\;\;\;\;\;\;\;\;x_{13}=f_{\gamma_{13}}(x_{11},x_{12})$\\

$k=3\;\;\;\;\;\;\;\;\;\;\;\;\;\;\;\;\;\;\;\;\;\;\;\;\;\;\;\;\;\;\;\;\;x_{7}=f_{\gamma_7}(x_3,x_6)\;\;\;\;\;\;\;\;\;\;\;\;\;\;\;\;\;\;\;\;\;\;\;\;\;\;\;\;\;\;\;\;\;\;\;\;\;\;\;\;\;\;\;\;\;\;\;\;x_{14}=f_{\gamma_{14}}(x_{10},x_{13})$\\

$k=4\;\;\;\;\;\;\;\;\;\;\;\;\;\;\;\;\;\;\;\;\;\;\;\;\;\;\;\;\;\;\;\;\;\;\;\;\;\;\;\;\;\;\;\;\;\;\;\;\;\;\;\;\;\;\;\;\;\;\;\;\;\;\;\;\;\;\;\;\;\;\;x_{15}=f_{\gamma_{15}}(x_{7},x_{14})$\\

\emph{and} $x_4=f_{\gamma_4}(x_3,x_3)$, $x_5=f_{\gamma_5}(x_3,x_3)$, $x_8=f_{\gamma_8}(x_7,x_7)$, $x_9=f_{\gamma_9}(x_7,x_7)$, $x_{11}=f_{\gamma_{11}}(x_{10},x_{10})$, $x_{12}=f_{\gamma_{12}}(x_{10},x_{10})$, and, for example, $x_{16}=f_{\gamma_{16}}(x_{15},x_{15})$.

\end{remark}
The following result gives a justification for the correctness of chaos game algorithm, which will be presented later. The idea base on the proof of a chaos game for IFSs presented in \cite{Martyn}.\\
If $p_1,...,p_m>0$ are such that $p_1+...+p_m=1$, then on the space $\Omega=\{1,...,m\}^\N$ we consider the product probability measure generated by a measure $P(\{i\})=p_i$ for $i\in\{1,...,m\}$ (see \cite{Mi} for deeper discussion connected with space of such measures endowed to GIFSs). 
\begin{theorem}\label{cg-main}
In the above setting, put $K_k:=\overline{\{x_i:i\geq k\}}$ for $k\in\N$. Then with probability $1$, we have that $K_k\to A_\F$ with respect to the Hausdorff-Pompeiu metric, and, moreover, if $x_0\in A_\F$, then $K_k=A_\F$ for every $k\in\N$.
\end{theorem}

\begin{proof}
We will just give the proof for the case $m=2$. The general case goes in the same way but is more technically complicated.\\
For any subsequence $(k_i)$ of naturals, consider the condition $(*)$ for a sequence $(\gamma_i)$:\\
$(*)\;\;$for every finite sequence $(\tilde{\gamma}_1,...,\tilde{\gamma}_l)$, there is infinitely many $i\in\N$ such that $\gamma_{k_i}=\tilde{\gamma}_1,...,\gamma_{k_i+l-1}=\tilde{\gamma}_l$.\\
By the Borel-Cantelli lemma (\cite{Fe}), condition $(*)$ is satisfied with a probability $1$ (i.e., there is a set $\mathcal{A}\subset \{1,...,n\}^\N$ such that $P(\mathcal{A})=1$, and for each $(\gamma_i)\in\mathcal{A}$, condition $(*)$ is satisfied).

For every $(j,k)\in\N\times\N$, we now define $\alpha_{(j,k)}\in\;_k\Omega$, inductively with respect to $k$.\\
So let $\alpha_{(j,1)}:=\gamma_{H(j,1)}$ for $j\in\N$
 and if $\alpha_{(j,k)}$ is defined for all $j\in\N$ and some $k\in\N$, then let 
$\alpha_{(j,k+1)}$ be such that $\alpha_{(j,k+1)}^1=\gamma_{H(j,k+1)}$ (i.e., the first coordinate is equal to $\gamma_{H(j,k+1)}$), and $\alpha_{(j,k+1)}(1)=\alpha_{(2j-1,k)}$ and $\alpha_{(j,k+1)}(2)=\alpha_{(2j,k)}$ (cf. (\ref{n1})).\\
\emph{For example, $\alpha_{(3,2)}=(\gamma_{10},(\gamma_8,\gamma_9))$ and $\alpha_{(1,3)}=(\gamma_7,(\gamma_{3},\gamma_6),((\gamma_1,\gamma_2),(\gamma_4,\gamma_5)))$, see Remark \ref{f3}).}\\
Assume first that $x_0\in A_\F$. 
By induction and the above construction it is easy to see that (recall the definition of $A_\alpha$ from earlier section):
\begin{equation}\label{f4}
x_{H(j,k)}\in A_{\alpha_{(j,k)}}\;\;\mbox{for all}\;\;(j,k)\in\N\times\N
\end{equation}
According to our observation from the beginning of the proof, with a probability $1$, each sequence $\alpha\in\Omega_k$ realizes as a sequence $\alpha_{(j,k)}$ infinitely many times (in a proper order - \emph{for example, to get $\alpha=(1,(2,3),((1,2),(4,3)))$, we have to choose $(1,2,2,4,3,3,1)$ - compare it with $\alpha_{(1,3)}$ above}) - we have to take, for example, a sequence $k_i=2^i$ (as we have to have more and more space). Hence, by (\ref{f4}), for every $\alpha\in\Omega_{<}$, there is infinitely many $i\in\N$ such that $x_i\in A_\alpha$. Thus, Lemma \ref{l1} implies that with a probability $1$, for every $k\in\N$, $\overline{\{x_i:i\geq k\}}=A_\F$, and the second assertion is proved.\\
Now assume that $y_0\in X$, and consider the sequence $(y_i)$ defined as the sequence $(x_i)$, but with $y_0$ as a starting point. In order to prove the first assertion, we only have to show that $d(x_i,y_i)\to 0$ (because then $H(\overline{\{x_i:i\geq k\}},\overline{\{y_i:i\geq k\}})\to 0$).\\
First we will prove that
\begin{itemize}
\item[(a)] $d\left(x_{H(j,k)},y_{H(j,k)}\right)\leq d(x_0,y_0)$ for all $(j,k)$;
\item[(b)] $d\left(x_{H(j,k)},y_{H(j,k)}\right)\leq \max\left\{\varphi^{(k-1)}\left(d(x_{H(2^{k-1}j-2^{k-1}+i,1)},y_{H(2^{k-1}j-2^{k-1}+i,1)}\right):\;1\leq i\leq 2^{k-1}\right\}$ for $(j,k)$;
\item[(c)] $d\left(x_{H(j,1)},y_{H(j,1)}\right)\leq \varphi\left(x_{H(1,{p+1})},y_{H(1,{p+1})}\right)$ where $j\geq 3$ and $p\geq 1$ is such that $2^p<j\leq 2^{p+1}$.
\end{itemize}
where $\varphi$ is a witness to the fact that $\F$ is a GIFS, and we also set $\varphi^0(t):=t$.\\
 \emph{ For example, $(b)$ says that $d(x_{14},y_{14})\leq \max\{\varphi^{(2)}(d(x_i,y_i)):i=8,9,11,12\}$ and $d(x_{15},y_{15})\leq \max\{\varphi^{(3)}(d(x_i,y_i)):i=1,2,4,5,8,9,11,12\}$,\\
and $(c)$ says that $d(x_5,y_5)\leq \varphi(d(x_3,y_3))$ and $d(x_{11},y_{11})\leq \varphi(d(x_7,y_7))$, see Remark \ref{f3}.}\\
The assertion $(a)$ is obvious. Assertion $(b)$ for $k=1$ is a trivial equality. Assume that $(b)$ holds for some $k\in\N$ and all $j\in\N$. We will prove it for $k+1$ and all $j\in\N$. Hence let $j\in\N$. We have
$$
d(x_{H(j,k+1)},y_{H(j,k+1)})=d(f_{\gamma_{H(j,k+1)}}(x_{H(2j-1,k)},x_{H(2j,k)}),f_{\gamma_{H(j,k+1)}}(y_{H(2j-1,k)},y_{H(2j,k)})\leq
$$
$$
\leq\varphi(\max\{d(x_{H(2j-1,k)},y_{H(2j-1,k)}),d(x_{H(2j,k)},y_{H(2j,k)})\})\leq
$$
$$
\leq \varphi(\max\{\max\{\varphi^{(k-1)}(d(x_{H(2^{k-1}(2j-1)-2^{k-1}+i,1)},y_{H(2^{k-1}(2j-1)-2^{k-1}+i,1)}):\;1\leq i\leq 2^{k-1}\},$$ $$\max\{\varphi^{(k-1)}(d(x_{H(2^{k-1}\cdot 2j-2^{k-1}+i,1)},y_{H(2^{k-1}\cdot 2j-2^{k-1}+i,1)}):\;1\leq i\leq 2^{k-1}\})=
$$
$$
=\max\{\varphi^{(k)}(d(x_{H(2^{k}j-2^k+i,1)},y_{H(2^{k}j-2^k+i,1)}):\;1\leq i\leq 2^{k}\}
$$
where inequality and equality follows from the fact that $\varphi$ is nondecreasing. Hence we get $(b)$. Now we prove $(c)$. If $j=2^p+s$ for $s=1,2$, then the assertion follows from definition (as $x_{H(2^p+s,1)}=f_{\gamma_{H(2^p+s,1)}}(x_{H(1,p+1)},x_{H(1,p+1)})$, and similarly $y_{H(2^p+s,1)}$ is defined). Now assume that $(c)$ holds for a fixed $p\in\N$ and $j=2^p+1,2^p+2,...,2l-1,2l$, where $2^p<2l<2^{p+1}$. Then for $s=1,2$ and proper $(j,k)$, we have by definition that $x_{H(2l+s,1)}=f_{\gamma_{H(2l+s,1)}}(x_{H(j,k)},x_{H(j,k)})$, and similarly $y_{H(2l+s,1)}$ is defined. Then by $(b)$, we have
$$
d(x_{H(2l+s,1)},y_{H(2l+s,1)})\leq\varphi\left(d(x_{H(j,k)},y_{H(j,k)}\right)\leq$$ $$\leq \varphi(\max\{\varphi^{(k-1)}(d(x_{H(2^{k-1}j-2^{k-1}+i,1)},y_{H(2^{k-1}j-2^{k-1}+i,1)}):\;1\leq i\leq 2^{k-1}\})\leq $$ $$\leq\varphi(x_{H(1,{p+1})},y_{H(1,{p+1})})
$$
as, by the construction and inductive assumption, each $$d(x_{H(2^{k-1}j-2^{k-1}+i,1)},y_{H(2^{k-1}j-2^{k-1}+i,1)})\leq \varphi(d(x_{H(1,p+1)},y_{H(1,p+1)}).$$ This gives us $(c)$.\\
We are ready to show that $d(x_i,y_i)\to 0$. We will use that fact that a sequence is convergent iff any of its subsequences has a convergent (to the same limit) subsequence.\\
Hence let $(r_i)$ be any subsequence of naturals, and for any $i\in\N$, let $(j_i,k_i)$ be such that $x_{r_i}=x_{H(j_i,k_i)}$ and $y_i=y_{H(j_i,k_i)}$. Consider two cases:\\
Case1: $\sup\{k_i:i\in\N\}=\infty$. Then switching to an appropriate subsequence, we can assume that $k_i\to\infty$ and hence, by $(a)$ and $(b)$, we have
$$
d(x_{r_i},y_{r_i})\leq\varphi^{(k_i-1)}(d(x_0,y_0))\to 0.
$$
Case2: $\sup\{k_i:i\in\N\}<\infty$. Then, switching to an appropriate subsequence, we can assume that $k_i=k$ for all $i\in\N$ and some $k\in\N$, and, consequently, $j_i\to\infty$. Then also $p_i\to\infty$, where $p_i$ is such that $2^{p_i}<j_i\leq 2^{p_i+1}$. Therefore, combining $(a)$, $(b)$ and $(c)$, we get
$$
d(x_{r_i},y_{r_i})\leq\varphi(x_{H(1,{p_i+1})},y_{H(1,{p_i+1})})\leq \varphi(\varphi^{(p_i)}(d(x_0,y_0)))=\varphi^{(p_i+1)}(d(x_0,y_0))\to 0
$$
All in all, $d(x_i,y_i)\to\infty$ and the proof is finished.
\end{proof}

Now we give a pseudcode of chaos game for GIFSs. It will not be as simple as for IFSs, since some points have to be remembered for a $"$long$"$ time (for example, to get $x_{14}$, we have to remembered $x_{10}$ and $x_{13}$, and to get $x_{15}$, we have to have $x_{7}$ and and $x_{14}$). Fortunately, it can be constructed so that only $m$ lists with reasonable and the same lengths will have to be remembered.
\begin{center}\textbf{Pseudocode for chaos game algorithm for GIFSs}\end{center}
Initially chosen point: $x_0\in X$.\\
Initially defined objects: constants: $m,n$, mappings: $f_1,...,f_n$, variables: $i,j,k$, list $x$ of the length $2$, lists: $z_1,..,z_m$.\\
Initial values: $j:=1$, $k:=1$, $z_1[0]:=x_0$,..., $z_m[0]:=x_0$.\\
First chosen point: choose randomly $\gamma\in\{1,...,n\}$

$\;\;\;\;\;\;\;\;\;\;\;\;\;\;\;\;\;\;\;\;\;\;\;\;\;\;$ $x[j,k]:=f_\gamma(z_1[0],...,z_m[0])$

$\;\;\;\;\;\;\;\;\;\;\;\;\;\;\;\;\;\;\;\;\;\;\;\;\;\;$ $z_1[1]:=x[j,k]$

$\;\;\;\;\;\;\;\;\;\;\;\;\;\;\;\;\;\;\;\;\;\;\;\;\;\;$ print $x[j,k]$.\newline\\
Main loop: $\;$ choose randomly $\gamma\in\{1,...,n\}$

$\;\;\;\;\;\;\;\;\;\;\;\;\;\;\;\;$ if $\;j\on{mod}m\neq 0$, then

$\;\;\;\;\;\;\;\;\;\;\;\;\;\;\;\;\;\;\;\;\;\;\;\;\;\;\;\;j:=m^{k-1}\cdot j+1$

$\;\;\;\;\;\;\;\;\;\;\;\;\;\;\;\;\;\;\;\;\;\;\;\;\;\;\;\;k:=1$

$\;\;\;\;\;\;\;\;\;\;\;\;\;\;\;\;\;\;\;\;\;\;\;\;\;\;\;\;x[j,k]:=f_\gamma(z_1[k-1],...,z_m[k-1])$

$\;\;\;\;\;\;\;\;\;\;\;\;\;\;\;\;\;\;\;\;\;\;\;\;\;\;\;\;$if $j\operatorname{mod}m=0$, then $i:=m$, else $i:=j\on{mod}m$

$\;\;\;\;\;\;\;\;\;\;\;\;\;\;\;\;\;\;\;\;\;\;\;\;\;\;\;\;z_i[k]:=x[j,k]$

$\;\;\;\;\;\;\;\;\;\;\;\;\;\;\;\;$ else

$\;\;\;\;\;\;\;\;\;\;\;\;\;\;\;\;\;\;\;\;\;\;\;\;\;\;\;\;j:=\frac{j}{m}$

$\;\;\;\;\;\;\;\;\;\;\;\;\;\;\;\;\;\;\;\;\;\;\;\;\;\;\;\;k:=k+1$

$\;\;\;\;\;\;\;\;\;\;\;\;\;\;\;\;\;\;\;\;\;\;\;\;\;\;\;\;x[j,k]:=f_\gamma(z_1[k-1],...,z_m[k-1])$

$\;\;\;\;\;\;\;\;\;\;\;\;\;\;\;\;\;\;\;\;\;\;\;\;\;\;\;\;$if $j\operatorname{mod}m=0$, then $i:=m$, else $i:=j\on{mod}m$

$\;\;\;\;\;\;\;\;\;\;\;\;\;\;\;\;\;\;\;\;\;\;\;\;\;\;\;\;z_i[k]:=x[j,k]$

$\;\;\;\;\;\;\;\;\;\;\;\;\;\;\;\;\;\;\;\;\;\;\;\;\;\;\;$ if $j\on{mod}m\neq 0$, then $z_1[0]:=x[j,k],$..., $z_m[0]:=x[j,k]$





$\;\;\;\;\;\;\;\;\;\;\;\;\;\;\;\;$ print $x[j,k]$.\\

Now we present examples:
\begin{example}\emph{
Consider the GIFSs $\F$ from Example \ref{e1}, and $\mH=\{h_1,h_2,h_3\}$, where for any $x=(x_1,x_2),y=(y_1,y_2)\in\R^2$, we have \\$\;$\\
$
h_1(x,y):=(0,25x_1+0,2y_2\;\mbox{\textbf{;}}\;0,25x_2+0,2y_2)
$\\
$
h_2(x,y):=(0,25x_1+0,2y_1\;\mbox{\textbf{;}}\;0,25x_2+0,1y_2+0,5)
$\\
$
h_3(x,y):=(0,25x_1+0,1y_1+0,5\;\mbox{\textbf{;}}\;0,25x_2+0,2y_2)
$\\$\;$\\ 
Using chaos game for GIFSs we get the following images
}
\begin{center}
\includegraphics{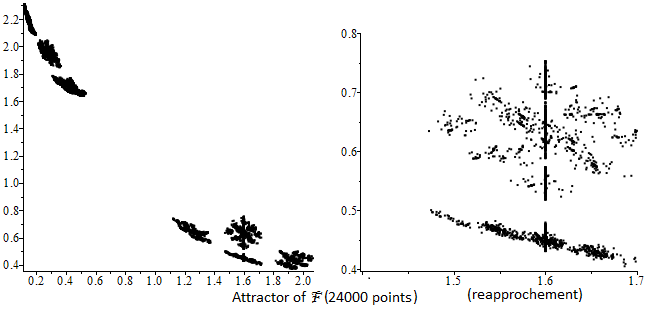}
\end{center}
\begin{center}
\includegraphics{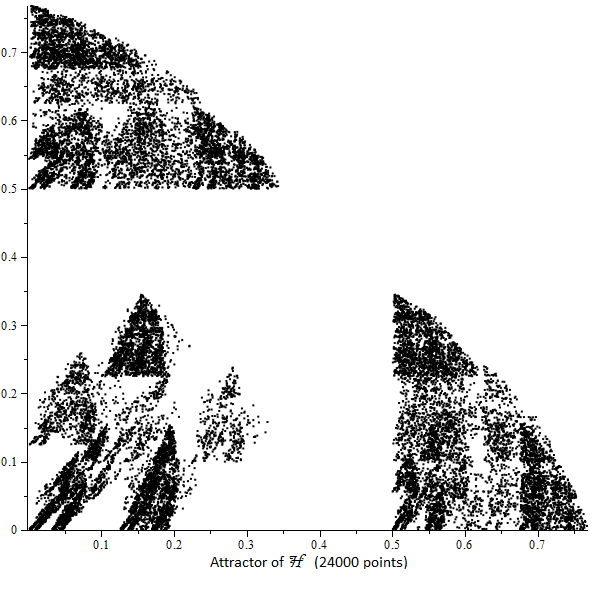}
\end{center}
\end{example}
\begin{remark}\emph{
Let us remark that lengths of lists $z_1,...,z_m$ are really reasonable - if we want to choose $\sum_{i=0}^{k-1}m^i=\frac{1-m^k}{1-m}$ points, then lengths of $z_1,...,z_m$ should be not less then $k$. Also, in this case, we need at most $km$ points to be remembered to proceed all calculations.
 Hence the above algorithm is indeed quite efficient.}
\end{remark}
\begin{remark}\emph{
It is easy to see that the algorithm gives us exactly the sequence $(x_i)$ (because $x[j,k]=x_{H(j,k)}$, and we get points in natural order $x_1,x_2,x_3,...$), with randomly chosen values $(\gamma_i)$.\\
We want also point out that the presented algorithm draw all points from the sequence $(x_i)$. However, it is obvious how to modify it in order to get the sequence $(x_i)$ without some of first of its points - such a change may be important if we start with $x_0$ which does not belong to $A_\F$. On the other hand, fixed points of mappings $f$ from a GIFS $\F$ belong to $A_\F$, so they are natural candidates for starting points.}
\end{remark}
\begin{remark}\emph{
There might be an attempt to define chaos game for GIFSs in alternative, even more similar to classical chaos game, way. Namely, for a given sequence $(\gamma_i)$, define the sequence $(x_i)$ in the following procedure:\\
 $x_0,...,x_{m-1}$ be chosen points from $X$, and $x_{k+m}:=f_{\gamma_{k+m}}(x_k,...,x_{k+m-1})$ for $k\geq 0$. \\
Is true that (with a probability $1$), sets $\{x_i:i\geq k\}$ are closer and closer to the attractor $A_\F$?\\
Unfortunately this does not work:\\
Let $x_0,x_1$ be any points from the attractor $A_\F$. Then (assuming the case $m=2$)
$$
x_6=f_{\gamma_6}(x_4,x_5)=f_{\gamma_6}\left(f_{\gamma_4}(x_2,x_3),f_{\gamma_5}(x_3,x_4)\right)=$$ $$=f_{\gamma_6}\left(f_{\gamma_4}\left(f_{\gamma_2}(x_0,x_1),f_{\gamma_3}(x_1,x_2)\right),f_{\gamma_5}\left(f_{\gamma_3}(x_1,x_2),f_{\gamma_4}(x_2,x_3)\right)\right)\in A_{(\gamma_6,(\gamma_4,\gamma_5),((\gamma_2,\gamma_3),(\gamma_3,\gamma_4)))}
$$
so some coefficients are the same. In particular, $x_6$ cannot belong to $A_\beta$ for $\beta:=(1,(1,1),((1,2),(1,2)))$. As $x_0,x_1$ were taken arbitrarily, this shows that for any $i\geq 6$, $x_i\notin A_\beta$, provided $A_\beta$ is disjoint with other sets $A_\alpha$, $\alpha\in\;_3\Omega$.\\
In the Cantor set defined in \cite{S}, the sets $A_\alpha$, $\alpha\in\;_3\Omega$ are pairwise disjoint, even in a stronger sense that for every $\alpha\in \;_3\Omega$, there is an open set $U_\alpha\subset \R^2$ such that $A_\alpha\subset U_\alpha$ and $U_\alpha's$ are pairwise disjoint (as it is standard construction of Cantor-like set; of course, the same holds for any given $k$, not necessarily $k=3$). Therefore, any set $\overline{\{x_i:i\geq 6\}}\cap A_\beta=\emptyset$ and, in particular, $\overline{\{x_i:i\geq 0\}}\neq\bigcup_{\alpha\in\;_3\Omega}A_\alpha=A_\F$ (see Lemma \ref{l1}).\\
In fact, we could take many other, much longer $\beta's$ - it seems that the set $\overline{\{x_n:n\geq 0\}}$ touches just selected sets $A_\alpha$.}
\end{remark}\label{e2}

\section{Affine GIFSs}
Observe that, having some GIFS $\F$, by Lemma \ref{filip4}, for every closed and bounded set $D$, sets $\bigcup_{\alpha\in\;_k\Omega}f_\alpha(D_k)$, $k\in\N$, are closer and closer to the attractor $A_\F$. Hence if we have a nice formula for mappings $f_\alpha$, $\alpha\in\Omega_<$, then we could easily get a good approximation of the attractor $A_\F$. In this section we will show that it can be done for a natural class of GIFSs.\\
We say that a GIFS $\F=\{f_1,...,f_n\}$ on an Euclidean space $\R^d$ is \emph{affine}, if each $f_i$ is an affine mapping, i.e., it is of the form:
\begin{equation}\label{aff}
f_i(x_1,...,x_m):=A^i_1\cdot x_1+...+A^i_m\cdot x_m+B^i,\;\;\;x_1,...,x_m\in\R^d
\end{equation}
where $A^i_1,...,A^i_m,B^i$ are some real square matrices of dimension $d$, and $\cdot$ and $+$ are standard multiplication and addition in space of matrices.\\
It turns out that in this case, mappings $f_\alpha$, $\alpha\in\Omega_<$ have natural descriptions.
Before we introduce it, let us give some further denotation.\\
If $X$ is a metric space, then let $X_1,X_2,...$ have a meaning as in (\ref{filip1}) and (\ref{filip2}).\\
If $x\in X_1=X\times...\times X$, then let $x_{(1)},...,x_{(m)}\in X$ be such that $x=(x_{(1)},...,x_{(m)})$.\\
Assume that for some $k\in\N$ and all $x\in X_k$, we defined $x_{(\epsilon_1,...,\epsilon_k)}\in X$, for $(\epsilon_1,...,\epsilon_k)\in\{1,...,m\}^k$.\\
If $x=(x_1,...,x_m)\in X_k\times...\times X_k=X_{k+1}$, then for every $(\epsilon_1,...,\epsilon_{k+1})\in\{1,...,m\}^{k+1}$, define $$x_{(\epsilon_1,...,\epsilon_{k+1})}:=(x_{\epsilon_1})_{(\epsilon_2,...,\epsilon_{k+1})}$$
All in all, for any $k\in\N$, any $x\in X_k$ and any $(\epsilon_1,...,\epsilon_k)\in\{1,...,m\}^k$, we defined $x_{(\epsilon_1,...,\epsilon_k)}\in X$.\\
In fact, we can give a bit less formal, but more natural description of elements $x_{(\epsilon_1,...,\epsilon_k)}$:\\
Observe that for every $k\in\N$, there is a natural bijection between $X_k$ and $X^{m^k}$, which adjust to any sequence $x\in X_k$, the sequence $\tilde{x}\in X^{m^k}$ which is created from $x$ by erasing all but the first and the last brackets.\\
\emph{For example, if $x=(((1,2),(3,1)),((2,1),(1,4)))$, then $\tilde{x}=(1,2,3,1,2,1,1,4)$.}\\
Now, we can enumerate elements of $\tilde{x}$ by using $m$-ary system, from the left to the right, but by using values $\{1,...,m\}$, instead of $\{0,...,m-1\}$. Then $x_{(\epsilon_1,...,\epsilon_k)}$ is exactly the $\epsilon_1,...,\epsilon_k$ term in a sequence $\tilde{x}$.\\
\emph{For example, $x=(((x_{(1,1,1)},x_{(1,1,2)}),(x_{(1,2,1)},x_{(1,2,2)})),((x_{(2,1,1)},x_{(2,1,2)}),(x_{2,2,1},x_{2,2,2})))$}.\\
In a similar way, for every $k\geq 2$ and $\alpha\in\;\Omega_k$, we define elements $\alpha_{(\epsilon_1,...,\epsilon_{k-1})}$.\\
\emph{For example, $\alpha=((\alpha_{(1,1)},\alpha_{(1,2)}),(\alpha_{(2,1)},\alpha_{(2,2)}))$.}\\
We are ready to state the theorem (recall that if $\alpha\in\;_k\Omega$, then $\alpha=(\alpha^1,...,\alpha^k)$ and $\alpha^i\in\Omega_i$ for $i=1,...,k$ - see denotations from the first section).

\begin{theorem}\label{it2}
Let $\F=\{f_1,...,f_n\}$ be an affine GIFS which consists of functions of the form (\ref{aff}).
For every $k\in\N$, $\alpha\in {_k\Omega}$ and $x\in X_k$, 
we have (all $\epsilon's$ below ranges from $1$ to $m$)
\begin{equation}\label{aff1}
f_{\alpha}(x) = 
\sum_{\epsilon_1, \epsilon_2, ..., \epsilon_k } A^{\alpha ^1}_{\epsilon_1} \cdot A^{\alpha^2_{(\epsilon_1)}}_{\epsilon_2} \cdot ... \cdot A^{\alpha^k_{(\epsilon_1, \epsilon_2, ..., \epsilon_{k-1})}}_{\epsilon_k} \cdot x_{(\epsilon_1, \epsilon_2, ..., \epsilon_k)} + {B^{\alpha}}  
\end{equation}
where
\begin{equation}\label{aff2}
{B^{\alpha}} := B^{\alpha^1} + \sum_{\epsilon_1} A^{\alpha^1}_{\epsilon_1}\cdot B^{\alpha^2_{(\epsilon_1)}} + ... + 
\sum_{\epsilon_1, \epsilon_2, ..., \epsilon_{k-1}} A^{\alpha ^1}_{\epsilon_1} \cdot A^{\alpha^2_{(\epsilon_1)}}_{\epsilon_2} \cdot ... \cdot A^{\alpha^{k-1}_{(\epsilon_1, \epsilon_2, ..., \epsilon_{k-2})}}_{\epsilon_{k-1}} \cdot B^{\alpha^k_{(\epsilon_1, \epsilon_2, ..., \epsilon_{k-1})}}
\end{equation}
\end{theorem}

\begin{proof}
The proof will follow by induction.

For $k=1$ let us take some $\alpha \in \{1, 2, ..., n\}$ and $x=(x_1, x_2, ..., x_m) \in X^m$. We have:
$$
f_{\alpha}(x) = A^{\alpha}_1 \cdot x_1+...+A^{\alpha}_m \cdot x_m+B^{\alpha} = \sum_{ \epsilon_1} A^{\alpha^1}_{\epsilon_1} \cdot x_{(\epsilon_1)} + B^{\alpha^1}
$$
so (\ref{aff1}) holds for $k=1$.\\
Now assume that it holds for some $k\geq 1$, and let us take some $\alpha=(\alpha^1,...,\alpha^{k},\alpha^{k+1}) \in {_{k+1}\Omega}$ and $x=(x_1,...,x_m)\in X_{k}\times...\times X_k=X_{k+1}$. We have (see (\ref{f1}) and other denotation from the first section).

$$f_{{\alpha}} (x) = f_{{\alpha}} (x_1,...,x_m) =
 f_{\alpha^1}\left(f_{{\alpha}(1)}(x_1),f_{{\alpha}(2)}(x_2),..., f_{{\alpha}(m)}(x_m)\right) = 
$$
$$
= A^{\alpha^1}_1 \cdot f_{{\alpha}(1)}(x_1) +A^{\alpha^1}_2 \cdot f_{{\alpha}(2)}(x_2) + ... +A^{\alpha^1}_m \cdot f_{{\alpha}(m)}(x_m) + B^{\alpha^1} = [\mbox{by definition and inductive assumption}]=
$$
$$
= A^{\alpha^1}_1 \cdot \left(\sum_{\epsilon_1, \epsilon_2, ..., \epsilon_k} A^{{\alpha}(1)^1}_{\epsilon_1} \cdot A^{{\alpha}(1)^2_{(\epsilon_1)}}_{\epsilon_2} \cdot ... \cdot A^{{\alpha}(1)^k_{(\epsilon_1, \epsilon_2, ..., \epsilon_{k-1})}}_{\epsilon_k} \cdot x_{(1, \epsilon_1, \epsilon_2, ..., \epsilon_k)} + {B^{{\alpha}{(1)}}}\right) + 
$$
$$
\ \ \ \ \ +A^{\alpha^1}_2 \cdot \left(\sum_{\epsilon_1, \epsilon_2, ..., \epsilon_k} A^{{\alpha}(2)^1}_{\epsilon_1} \cdot A^{{\alpha}(2)^2_{(\epsilon_1)}}_{\epsilon_2} \cdot ... \cdot A^{{\alpha}(2)^k_{(\epsilon_1, \epsilon_2, ..., \epsilon_{k-1})}}_{\epsilon_k} \cdot x_{(2, \epsilon_1, \epsilon_2, ..., \epsilon_k)} +{B^{{\alpha}{(2)}}}\right) + ... + 
$$
$$
\ \ \ \ \ \ \ \ \ \ \ \ + A^{\alpha^1}_m \cdot \left(\sum_{\epsilon_1, \epsilon_2, ..., \epsilon_k} A^{{\alpha}(m)^1}_{\epsilon_1} \cdot A^{{\alpha}(m)^2_{(\epsilon_1)}}_{\epsilon_2} \cdot ... \cdot A^{{\alpha}(m)^k_{(\epsilon_1, \epsilon_2, ..., \epsilon_{k-1})}}_{\epsilon_k} \cdot x_{(m, \epsilon_1, \epsilon_2, ..., \epsilon_k)} + {B^{{\alpha}{(m)}}}\right) + B^{\alpha^1} = \bigotimes
$$
Now, we should observe that for any $i=1, 2, ..., m$ and any $s=1,...,k$, $\alpha(i)^s = \alpha^{s+1}_i$, provided $\alpha=(\alpha^1,(\alpha^2_1,...,\alpha^2_m),...,(\alpha^{k+1}_1,...,\alpha^{k+1}_m))$. With this observation we come to:
$$
\bigotimes= A^{\alpha^1}_1 \cdot \left(\sum_{\epsilon_1, \epsilon_2, ..., \epsilon_k} A^{{\alpha}^2_1}_{\epsilon_1} \cdot A^{({\alpha}^3_1)_{(\epsilon_1)}}_{\epsilon_2} \cdot ... \cdot A^{({\alpha}^{k+1}_1)_{(\epsilon_1, \epsilon_2, ..., \epsilon_{k-1})}}_{\epsilon_n} \cdot x_{(1, \epsilon_1, \epsilon_2, ..., \epsilon_k)} + {B^{{\alpha}{(1)}}}\right) + 
$$
$$
\ \ \ \ \ \ \ A^{\alpha^1}_2 \cdot \left(\sum_{\epsilon_1, \epsilon_2, ..., \epsilon_k} A^{{\alpha}^2_2}_{\epsilon_1} \cdot A^{({\alpha}^3_2)_{(\epsilon_1)}}_{\epsilon_2} \cdot ... \cdot A^{({\alpha}^{k+1}_2)_{(\epsilon_1, \epsilon_2, ..., \epsilon_{k-1})}}_{\epsilon_n} \cdot x_{(2, \epsilon_1, \epsilon_2, ..., \epsilon_k)} + {B^{{\alpha}{(2)}}}\right) + ... + 
$$
$$
\ \ \ \ \ \ \ \ \ \ \ \ \ A^{\alpha^1}_m \cdot \left(\sum_{\epsilon_1, \epsilon_2, ..., \epsilon_k} A^{{\alpha}^2_m}_{\epsilon_1} \cdot A^{({\alpha}^3_m)_{(\epsilon_1)}}_{\epsilon_2} \cdot ... \cdot A^{({\alpha}^{k+1}_m)_{(\epsilon_1, \epsilon_2, ..., \epsilon_{k-1})}}_{\epsilon_n} \cdot x_{(m, \epsilon_1, \epsilon_2, ..., \epsilon_k)} + {B^{{\alpha}{(m)}}}\right) + B^{\alpha^1}= \bigotimes
$$
Now observe that for any $i=1,...,m$ and $s=3,...,k+1$, and every $\epsilon_1,...,\epsilon_{s-2}$, we have $(\alpha^s_i)_{(\epsilon_1,...,\epsilon_{s-2})}=\alpha^s_{(i,\epsilon_1,...,\epsilon_{s-2})}$, and, similarly, $\alpha^2_i=\alpha^2_{(i)}$. Hence
$$
\bigotimes= A^{\alpha^1}_1 \cdot \left(\sum_{\epsilon_1, \epsilon_2, ..., \epsilon_k} A^{{\alpha}^2_{(1)}}_{\epsilon_1} \cdot A^{{\alpha}^3_{(1, \epsilon_1)}}_{\epsilon_2} \cdot ... \cdot A^{{\alpha}^{k+1}_{(1, \epsilon_1, \epsilon_2, ..., \epsilon_{k-1})}}_{\epsilon_k} \cdot x_{(1, \epsilon_1, \epsilon_2, ..., \epsilon_k)} + {B^{{\alpha}{(1)}}}\right) + 
$$
$$
\ \ \ \ \ \ \ \ A^{\alpha^1}_2 \cdot \left(\sum_{\epsilon_1, \epsilon_2, ..., \epsilon_k} A^{{\alpha}^2_{(2)}}_{\epsilon_1} \cdot A^{{\alpha}^3_{(2, \epsilon_1)}}_{\epsilon_2} \cdot ... \cdot A^{{\alpha}^{k+1}_{(2, \epsilon_1, \epsilon_2, ..., \epsilon_{k-1})}}_{\epsilon_k} \cdot x_{(2, \epsilon_1, \epsilon_2, ..., \epsilon_k)} + {B^{{\alpha}{(2)}}}\right) + ... + 
$$
$$
\ \ \ \ \ \ \ \ \ \ \ \ \ \  A^{\alpha^1}_m \cdot \left(\sum_{\epsilon_1, \epsilon_2, ..., \epsilon_k} A^{{\alpha}^2_{(m)}}_{\epsilon_1} \cdot A^{{\alpha}^3_{(m, \epsilon_1)}}_{\epsilon_2} \cdot ... \cdot A^{{\alpha}^{k+1}_{(m, \epsilon_1, \epsilon_2, ..., \epsilon_{k-1})}}_{\epsilon_k} \cdot x_{(m, \epsilon_1, \epsilon_2, ..., \epsilon_k)} + {B^{{\alpha}{(m)}}}\right) + B^{\alpha^1} =\bigotimes
$$
Now, remunerating $\epsilon's$ and rearranging coefficients, we get
$$
\bigotimes= \sum_{\epsilon_1, \epsilon_2, \epsilon_3, ..., \epsilon_{k+1}} A^{\alpha^1}_{\epsilon_1} \cdot A^{{\alpha}^2_{(\epsilon_1)}}_{\epsilon_2} \cdot A^{{\alpha}^3_{(\epsilon_1, \epsilon_2)}}_{\epsilon_3} \cdot ... \cdot A^{{\alpha}^k_{(\epsilon_1, \epsilon_2, \epsilon_3, ..., \epsilon_k)}}_{\epsilon_{k+1}} \cdot x_{(\epsilon_1, \epsilon_2, \epsilon_3, ..., \epsilon_{k+1})}+
$$
$$
+ A^{\alpha^1}_1 \cdot {B^{{\alpha}{(1)}}} + A^{\alpha^1}_2 \cdot {B^{{\alpha}{(2)}}}+...+ A^{\alpha^1}_m \cdot {B^{{\alpha}{(m)}}}+B^{\alpha^1}
$$
Hence it remains to show that 
\begin{equation}\label{fifi1}
B^\alpha=A^{\alpha^1}_1 \cdot {B^{{\alpha}{(1)}}} + A^{\alpha^1}_2 \cdot {B^{{\alpha}{(2)}}}+...+ A^{\alpha^1}_m \cdot {B^{{\alpha}{(m)}}}+B^{\alpha^1}
\end{equation}
We have (we will use the same tricks as earlier without mentioning them)
$$
A^{\alpha^1}_1 \cdot {B^{{\alpha}{(1)}}} + A^{\alpha^1}_2 \cdot {B^{{\alpha}{(2)}}}+...+ A^{\alpha^1}_m \cdot {B^{{\alpha}{(m)}}}+B^{\alpha^1}= \sum_{\epsilon_1} A^{\alpha^1}_{\epsilon_1} \cdot {B^{{\alpha}(\epsilon_1)}} + B^{\alpha^1} = [\mbox{by inductive assumption}]=
$$
$$
=B^{\alpha^1}+ \sum_{\epsilon_1} A^{\alpha^1}_{\epsilon_1} \cdot \left(B^{{\alpha}(\epsilon_1)^1} + \sum_{\epsilon_2} A^{{\alpha}(\epsilon_1)^1}_{\epsilon_2} \cdot B^{{\alpha}(\epsilon_1)^2_{(\epsilon_2)}} + ...
   + \sum_{\epsilon_2, ..., \epsilon_k}\cdot A^{{\alpha}(\epsilon_1)^1}_{\epsilon_2} \cdot ... \cdot A^{{\alpha}(\epsilon_1)^{k-1}_{(\epsilon_2, ..., \epsilon_{k-1})}}_{\epsilon_k} \cdot B^{{\alpha}(\epsilon_1)^k_{(\epsilon_2, ..., \epsilon_k)}}\right) = 
$$
$$
= B^{\alpha^1}+ \sum_{\epsilon_1} A^{\alpha^1}_{\epsilon_1} \cdot \left(B^{\alpha^2_{\epsilon_1}} + \sum_{\epsilon_2} A^{\alpha^2_{\epsilon_1}}_{\epsilon_2} \cdot B^{(\alpha^3_{\epsilon_1})_{(\epsilon_2)}} + ... + \sum_{\epsilon_2, ..., \epsilon_k} A^{\alpha^2_{\epsilon_1}}_{\epsilon_2} \cdot ... \cdot A^{(\alpha^k_{\epsilon_1})_{(\epsilon_2, ..., \epsilon_{k-1})}}_{\epsilon_k} \cdot B^{(\alpha^{k+1}_{\epsilon_1})_{(\epsilon_2, ..., \epsilon_k)}}\right) =  
$$
$$
=B^{\alpha^1}+ \sum_{\epsilon_1} A^{\alpha^1}_{\epsilon_1} \cdot B^{\alpha^2_{(\epsilon_1)}} + \sum_{\epsilon_1,\epsilon_2} A^{\alpha^1}_{\epsilon_1} \cdot A^{\alpha^2_{(\epsilon_1)}}_{\epsilon_2} \cdot B^{\alpha^3_{(\epsilon_1, \epsilon_2)}} + ... + \sum_{\epsilon_1, ..., \epsilon_k} A^{\alpha^1}_{\epsilon_1} \cdot A^{\alpha^2_{(\epsilon_1)}}_{\epsilon_2} \cdot ... \cdot A^{\alpha^k_{(\epsilon_1, \epsilon_2, ..., \epsilon_{k-1})}}_{\epsilon_k} \cdot B^{\alpha^{k+1}_{(\epsilon_1, \epsilon_2, ..., \epsilon_k)}}) =  
 {B^{{\alpha}}}.
$$
\end{proof}

Now we are going to present a pseudocode of the algorithm, which will generate the values
\begin{equation}\label{baff3}
A^\alpha_{\epsilon_1,...,\epsilon_k}:=A^{\alpha ^1}_{\epsilon_1} \cdot A^{\alpha^2_{(\epsilon_1)}}_{\epsilon_2} \cdot ... \cdot A^{\alpha^k_{(\epsilon_1, \epsilon_2, ..., \epsilon_{k-1})}}_{\epsilon_k} 
\end{equation}
and
\begin{equation}\label{baff4}
B^\alpha:=B^{\alpha^1} + \sum_{\epsilon_1} A^{\alpha^1}_{\epsilon_1}\cdot B^{\alpha^2_{(\epsilon_1)}} + ... + 
\sum_{\epsilon_1, \epsilon_2, ..., \epsilon_{k-1}} A^{\alpha ^1}_{\epsilon_1} \cdot A^{\alpha^2_{(\epsilon_1)}}_{\epsilon_2} \cdot ... \cdot A^{\alpha^{k-1}_{(\epsilon_1, \epsilon_2, ..., \epsilon_{k-2})}}_{\epsilon_{k-1}} \cdot B^{\alpha^k_{(\epsilon_1, \epsilon_2, ..., \epsilon_{k-1})}}
\end{equation}

As we want to work with numbers rather than quite many lists, we need to define some further denotations and remarks. Assume that $\F=\{f_1,...,f_n\}$ is a GIFS of order $m$.\\
At first observe that for any $k\in\N$, we can adjust to each $\alpha\in\;\Omega_k$ (and $\alpha\in\;_k\Omega$), the sequence $\tilde{\alpha}$ of the length $m^{k-1}$ (and $1+...+m^{k-1}=\frac{1-m^k}{1-m}$ respectively), in a similar way as we did it for $x's$ -earlier, i.e., by erasing all brackets but the first and the last one.\\
\emph{for example, if $\alpha=(1,(2,1),((3,2),(4,1)))$, then $\tilde{\alpha}=(1,2,1,3,2,4,1)$.}\\
Therefore ($\on{card}(\cdot)$ denotes the cardinality of a set)
\begin{equation}\label{afff1}
\on{card}(\Omega_k)=n^{m^{k-1}}\;\;\;\mbox{and}\;\;\;\on{card}(_k\Omega)=n^{\frac{1-m^k}{1-m}}
\end{equation}
Now to any $\alpha\in\;_k\Omega$, we can adjust the number $N(\alpha,k)$ such that the sequence $\tilde{\alpha}$ is its representation in the $n$-ary system, but with the the use of numbers $1,...,n$ instead of $0,...,n-1$. In a similar way we adjust to each $\alpha\in\;\Omega_k$, the number $M(\alpha,k)$.\\
\emph{For example, for $\alpha$ as above (and for $n=4$), we have (note that here $6=\frac{1-2^3}{1-2}-1=\frac{1-m^k}{1-m}-1$) 
$$N(\alpha,3)= 0\cdot 4^6+1\cdot 4^5+0\cdot 4^4+2\cdot 4^3+1\cdot 4^2+3\cdot 4^1+0\cdot 4^0=$$
$$=(1-1)\cdot 4^6+(2-1)\cdot 4^5+(1-1)\cdot 4^4+(3-1)\cdot 4^3+(2-1)\cdot 4^2+(4-1)\cdot 4^1+(1-1)\cdot 4^0$$}
Clearly, the mappings $\alpha\to N(\alpha,k)$ and $\alpha\to M(\alpha,k)$ are bijections of  $_k\Omega$ and $\left\{0,...,n^{\frac{1-m^k}{1-m}}-1\right\}$, and $\Omega_k$ and $\left\{0,...,n^{m^{k-1}}-1\right\}$, respectively.\\
Also, by definition, if $\alpha\in\;_{k-1}\Omega$ and $\gamma\in\Omega_k$, then
\begin{equation}\label{afff2}
N(\alpha\hat\;\gamma,k)=N(\alpha,k-1)\cdot n^{m^{k-1}}+M(\gamma,k),
\end{equation}
where $\alpha\hat\;\gamma$ denotes the extension of $\alpha$ by $\gamma$.
Similarly, to any sequence $\epsilon=(\epsilon_1,...,\epsilon_k)\in\{1,...,m\}^k$, we adjust the number 
\begin{equation}\label{aa1}P(\epsilon,k):=(\epsilon_1-1)\cdot m^{k-1}+...+(\epsilon_{k-1}-1)\cdot m^1+(\epsilon_{k}-1)\cdot m^0
\end{equation}
Clearly, if $\epsilon=(\epsilon_1,...,\epsilon_k)$ and $\epsilon'=(\epsilon_1,...,\epsilon_{k-1})$, then \begin{equation}\label{aa2}P(\epsilon,k)=m\cdot P(\epsilon',k-1)+\epsilon_k-1\end{equation}
Now we show that if $\alpha\in\Omega_k$ and $\epsilon=(\epsilon_1,...,\epsilon_{k-1})\in\{1,...,m\}^{k-1}$, then
\begin{equation}\label{aff3}
\alpha_{(\epsilon_1,...,\epsilon_{k-1})}=\left(\on{floor}\left(\frac{M(\alpha,k)}{n^{m^{k-1}-1-P(\epsilon,k-1)}}\right)\right)\on{mod}n+1
\end{equation}
where $\on{floor}$ is the integer part of a number.\\
To see (\ref{aff3}), observe that $\alpha_{(\epsilon_1,...,\epsilon_{k-1})}$ is exactly the $m^{k-1}-P(\epsilon,k-1)$ term (from the right) in a sequence $\tilde{\alpha}$ (which can be proved by induction with respect to $k$, with the use of (\ref{aa2})), and the formula in (\ref{aff3}) detects this term from $\tilde{\alpha}$ (we need to add $1$ because we use numbers $1,...,m$ instead of $0,...,m-1$).\\
\emph{For example, let $\alpha=(((1,2),(4,3)),((3,2),(1,2)))$ (and $n=4)$, and $\epsilon=(1,2,1)$. Then $\alpha_{(1,2,1)}=4$, and our formula (\ref{aff3}) gives us the same: $M(\alpha,4)=7825$, $P(\epsilon,3)=2$, and $\left(\on{floor}\left(\frac{785}{4^{2^{4-1}-1-2}}\right)\right)\on{mod}n+1=4$.
}\\
Now, for every $\alpha=(\alpha^1,...,\alpha^k)\in\;_k\Omega$ and $\epsilon=(\epsilon_1,...,\epsilon_k)\in\{1,...,m\}^k$, denote
$$
A[k,N(\alpha,k),P(\epsilon,k)]:=A^{\alpha}_{\epsilon_1,...,\epsilon_k}$$
$$B[k,N(\alpha,k)]:=B^\alpha$$
and if $k\geq 2$, then
$$C[k,N(\alpha,k),P(\epsilon',k-1)]:=A^{\alpha'}_{\epsilon_1,...,\epsilon_{k-1}}\cdot B^{\alpha^k_{(\epsilon_1,...,\epsilon_{k-1})}}
$$
where $\alpha'=(\alpha^1,...,\alpha^{k-1})$ and $\epsilon'=(\epsilon_1,...,\epsilon_{k-1})$. Then by (\ref{baff3}), (\ref{baff4}), (\ref{aa1}), (\ref{aa2}) and (\ref{afff2}), we have for $k\geq 2$:
\begin{equation}\label{aa3}
A[k,N(\alpha,k),P(\epsilon,k)]=A\left[k,N(\alpha',k-1)\cdot n^{m^{k-1}}+M(\alpha^k,k),P(\epsilon',k-1)\cdot m+\epsilon_k-1\right]=\end{equation} $$=A[k-1,N(\alpha',k-1),P(\epsilon',k-1)]\cdot A\left[1,\alpha^k_{(\epsilon_1,...,\epsilon_{k-1})}-1,\epsilon_{k}-1\right]
$$
and
\begin{equation}\label{aa4}
C[k,N(\alpha,k),P(\epsilon',k-1)]=A[k-1,N(\alpha',k-1),P(\epsilon',k-1)]\cdot B\left[1,\alpha^k_{(\epsilon_1,...,\epsilon_{k-1})}-1\right]
\end{equation}
and, finally,
\begin{equation}\label{aa5}
B[k,N(\alpha,k)]=B[k-1,N(\alpha',k-1)]+\sum_{\epsilon=(\epsilon_1,...,\epsilon_{k-1})}C[k,N(\alpha,k),P(\epsilon,k-1)]
\end{equation}
We are ready to give a pseudocode of an algorithm which will generate values (which are, in fact, matrices) $A[k,N,P]$, $C[k,N,P]$ and $B[k,N]$, according to the above procedure. Then, having these values, we automatically have the mappings $f_\alpha$ and hence, using Lemma \ref{filip4}, we can easily make images of $A_\F$. Note that in the main loop below we used (\ref{aa3}), (\ref{aa4}) and (\ref{aa5}).
\newpage
\begin{center}\textbf{Pseudocode for affine GIFSs - defining all coefficients}\end{center}
Initially defined objects: constants: $m,n,A^i_j,B^i$, $i=1,...,n$, $j=1,...,m$, variables: $N,M,P,I,k$, matrices: $A,B,C$.\\

First defined values:

$\;\;\;\;\;\;\;\;\;\;\;\;\;\;\;\;\;\;\;\;\;$ $k:=1$

$\;\;\;\;\;\;\;\;\;\;\;\;\;\;\;\;\;\;\;\;\;$ For $N$ from $0$ to $n-1$

$\;\;\;\;\;\;\;\;\;\;\;\;\;\;\;\;\;\;\;\;\;\;\;\;\;\;\;$ $B[k,N]:=B^{N+1}$

$\;\;\;\;\;\;\;\;\;\;\;\;\;\;\;\;\;\;\;\;\;\;\;\;\;\;\;$ For $P$ from $0$ to $m-1$

$\;\;\;\;\;\;\;\;\;\;\;\;\;\;\;\;\;\;\;\;\;\;\;\;\;\;\;\;\;\;\;\;$ $A[k,N,P]:=A^{N+1}_{P+1}$

Main loop

$\;\;\;\;\;\;\;\;\;\;\;\;\;\;\;\;\;\;\;\;\;$ $k:=k+1$

$\;\;\;\;\;\;\;\;\;\;\;\;\;\;\;\;\;\;\;\;\;$ For $N$ from $0$ to $n^{\frac{m^{k-1}-1}{m-1}}-1$

$\;\;\;\;\;\;\;\;\;\;\;\;\;\;\;\;\;\;\;\;\;\;\;\;\;\;\;$ For $M$ from $0$ to $n^{m^{k-1}}-1$

$\;\;\;\;\;\;\;\;\;\;\;\;\;\;\;\;\;\;\;\;\;\;\;\;\;\;\;\;\;\;\;\;$ $B\left[k,N\cdot n^{m^{k-1}}+M\right]:=B[k-1,N]$

$\;\;\;\;\;\;\;\;\;\;\;\;\;\;\;\;\;\;\;\;\;\;\;\;\;\;\;\;\;\;\;\;$ For $P$ from $0$ to $m^{k-1}-1$

$\;\;\;\;\;\;\;\;\;\;\;\;\;\;\;\;\;\;\;\;\;\;\;\;\;\;\;\;\;\;\;\;\;\;\;\;\;$ $C\left[k,N\cdot n^{m^{k-1}}+M,P\right]:=$

$\;\;\;\;\;\;\;\;\;\;\;\;\;\;\;\;\;\;\;\;\;\;\;\;\;\;\;\;\;\;\;\;\;\;\;\;\;\;\;\;\;\;\;\;\;\;\;\;\;\;\;\;\;$
$:=A[k-1,N,P]\cdot B\left[1,\left(\on{floor}\left(\frac{M}{n^{m^{k-1}-1-P}}\right)\right)\on{mod}n\right]$

$\;\;\;\;\;\;\;\;\;\;\;\;\;\;\;\;\;\;\;\;\;\;\;\;\;\;\;\;\;\;\;\;\;\;\;\;\;$ $B\left[k,N\cdot n^{m^{k-1}}+M\right]:=B\left[k,N\cdot n^{m^{k-1}}+M\right]+C\left[k,N\cdot n^{m^{k-1}}+M,P\right]$

$\;\;\;\;\;\;\;\;\;\;\;\;\;\;\;\;\;\;\;\;\;\;\;\;\;\;\;\;\;\;\;\;\;\;\;\;\;$ For $I$ form $0$ to $m-1$

$\;\;\;\;\;\;\;\;\;\;\;\;\;\;\;\;\;\;\;\;\;\;\;\;\;\;\;\;\;\;\;\;\;\;\;\;\;\;\;\;\;\;\;\;$ 
$A\left[k,N\cdot n^{m^{k-1}}+M,P\cdot m+I\right]:=$

$\;\;\;\;\;\;\;\;\;\;\;\;\;\;\;\;\;\;\;\;\;\;\;\;\;\;\;\;\;\;\;\;\;\;\;\;\;\;\;\;\;\;\;\;\;\;\;\;\;\;\;\;\;$
 $:=A[k-1,N,P]\cdot A\left[1,\left(\on{floor}\left(\frac{M}{n^{m^{k-1}-1-P}}\right)\right)\on{mod}n,I\right]$

%

$\;\;\;\;\;\;\;\;\;\;\;\;\;\;\;\;\;\;\;\;\;\;\;\;\;\;\;\;\;\;\;\;\;\;\;\;\;\;\;\;\;\;\;\;$

\begin{remark}\emph{
Note that for computing $B^\alpha$ and $A's$ on the step $k$, we need $n+n^{\frac{m^{k-1}-1}{m-1}}\left(1+m^{k-1}\right)$ places for matrices from the first and the $k-1$ steps, which are necessary for further computations, and $n^{\frac{m^k-1}{m-1}}\left(1+(m+1)m^{k-1}\right)$ places for a result.
}
\end{remark}

The algorithm given above is quite expansive, even if we make it a bit more efficient (for example, we can avoid making some computations many times by defining certain lists at the beginning).

Therefore we also present a shortcut. The crucial observation is that if in Lemma \ref{filip4} we take $D:=\{0\}$, then for every $\alpha\in\;_k\Omega$, $f_\alpha(D_k)=\{B^\alpha\}$ (because in this case, $x_{(\epsilon_1,...,\epsilon_k)}=0$). Hence sets $\{B^\alpha:\alpha\in\;_k\Omega\}$ are better and better approximations of the attractor $A_\F$. Also, by a proof of Theorem \ref{it2} (see (\ref{fifi1})), we have
$$
B^\alpha=A^{\alpha^1}_1 \cdot {B^{{\alpha}{(1)}}} + A^{\alpha^1}_2 \cdot {B^{{\alpha}{(2)}}}+...+ A^{\alpha^1}_m \cdot {B^{{\alpha}{(m)}}}+B^{\alpha^1}
$$
for all $\alpha\in\;_k\Omega$, $k\geq 2$. This gives a suggestion how appropriate algorithm should work (of course, from the previous one we also obtain elements $B^\alpha$).\\
Let $\alpha=(\alpha^1,(\alpha^2_1,...,\alpha^2_m),...(\alpha^k_1,...,\alpha^k_m))\in\;_k\Omega$, $k\geq 2$. By similar reasonings as earlier, we can see that for any $i=2,...,k$ and $j=1,...,m$, we have
\begin{equation}\label{fifi2}
M(\alpha^i_j,i-1)=\left(\on{floor}\left(\frac{N(\alpha,k)}{n^{\frac{1-m^k}{1-m}-\frac{1-m^{i-1}}{1-m}-j\cdot m^{i-2}}}\right)\right)
\on{mod}n^{m^{i-2}}
\end{equation}
and
\begin{equation}\label{fifi3}
N(\alpha(j),k-1)=M\left(\alpha^k_j,k-1\right)+M\left(\alpha^{k-1}_j,k-2\right)\cdot n^{m^{k-2}}+
M\left(\alpha^{k-2}_j,k-3\right)\cdot n^{m^{(k-2)}+m^{(k-3)}}+...+\end{equation} $$+...+M\left(\alpha^2_j,1\right)\cdot n^{m^{(k-2)}+m^{(k-3)}+...+m^{1}}
$$
provided $k>2$, and
\begin{equation}\label{fifi4}
N(\alpha(j),1)=M\left(\alpha^2_j,1\right)
\end{equation}
provided $k=2$. Observe that (\ref{fifi3}) and (\ref{fifi4}) can be written in this form:
$$
N(\alpha(j),k-1)=\sum_{i=2}^{k}M(\alpha^i_j,i-1)n^{\frac{m^{i-1}-m^{k-1}}{1-m}}
$$

We are ready to give the pseudocode for the algorithm. $N$ will mean $N(\alpha,k)$, $B[k,N]$ will mean $B[k,N(\alpha,k)]$, and $P[j]$ will mean $N[\alpha(j),k-1]$.

\begin{center}\textbf{Pseudocode for affine GIFSs - defining variables $B^\alpha$}\end{center}

Initially defined objects: constants $m,n,A^i_j,B^i$, $i=1,...,n$, $j=1,...,m$, variables: $N,i,j,k,M$, list: $P$, matrices: $A,B$

Initially defined values:

$\;\;\;\;\;\;\;\;\;\;\;\;\;\;$

$\;\;\;\;\;\;\;\;\;\;\;\;\;\;$ $k:=1$

$\;\;\;\;\;\;\;\;\;\;\;\;\;\;$ For $N$ from $0$ to $n-1$

$\;\;\;\;\;\;\;\;\;\;\;\;\;\;\;\;\;\;\;$ $B[k,N]:=B^{N+1}$

$\;\;\;\;\;\;\;\;\;\;\;\;\;\;\;\;\;\;\;$ For $j$ from $1$ to $m$

$\;\;\;\;\;\;\;\;\;\;\;\;\;\;\;\;\;\;\;\;\;\;\;\;\;$ $A[N,j]:=A^{N+1}_j$




Main loop:

$\;\;\;\;\;\;\;\;\;\;\;\;\;\;$ $k:=k+1$

$\;\;\;\;\;\;\;\;\;\;\;\;\;\;$ For $N$ form $0$ to $n^{\frac{m^k-1}{m-1}}-1$

$\;\;\;\;\;\;\;\;\;\;\;\;\;\;\;\;\;\;\;\;$ For $j$ from $1$ to $m$

$\;\;\;\;\;\;\;\;\;\;\;\;\;\;\;\;\;\;\;\;\;\;\;\;\;\;$ $P[j]:=\sum_{i=2}^k\left(\left(\on{floor}\left(\frac{N}{n^{\frac{1-m^k}{1-m}-\frac{1-m^{i-1}}{1-m}-j\cdot m^{i-2}}}\right)\right)
\on{mod}n^{m^{i-2}}\right)\cdot n^{\frac{m^{i-1}-m^{k-1}}{1-m}}$




$\;\;\;\;\;\;\;\;\;\;\;\;\;\;\;\;\;\;\;\;$ $B[k,N]:=\sum_{j=1}^mA\left[\on{floor}\left(\frac{N}{n^{\frac{m^{k}-1}{m-1}-1}}\right),j\right]\cdot B\left[k-1,P[j]\right]+B\left[1,\on{floor}\left(\frac{N}{n^{\frac{m^{k}-1}{m-1}-1}}\right)\right]$\\

Finally, we present the example.
\begin{example}\label{e3}\emph{
Consider the GIFSs $\F$ and $\G$ from Example \ref{e1}. Clearly, they are affine. Using the above algorithm, we get the following images:}
\begin{center}
\includegraphics{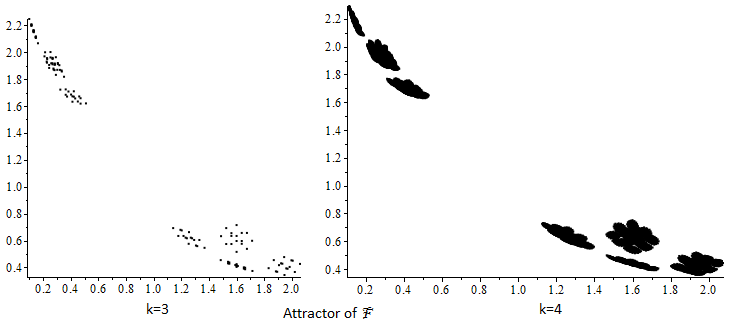}
\end{center}
\begin{center}
\includegraphics{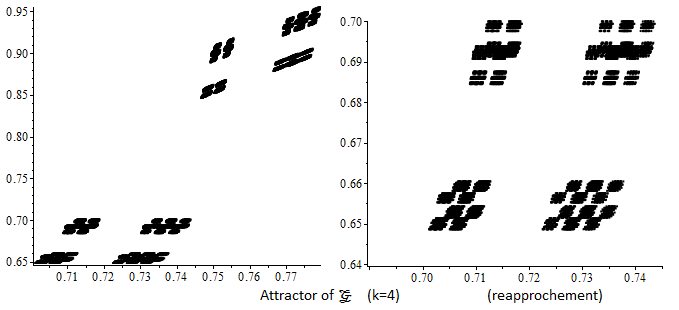}
\end{center}
\end{example}

\begin{remark}\emph{
Note that in the "shorter" version of the algorithm, for computing $B^\alpha$ on the step $k$, we need $n^{\frac{m^{k-1}-1}{m-1}}+n(m+1)$ places for matrices from the first and the $k-1$ steps, which are necessary for further computations, and $m+n^{\frac{m^k-1}{m-1}}$ places for a result.
}
\end{remark}

\end{document}